\documentclass[12 pt]{amsart}
\usepackage[all]{xy}

\usepackage[urlcolor=blue]{hyperref} 

\usepackage{enumitem} 
\usepackage{xspace} 

\usepackage{amsfonts}
\usepackage{amsmath,amssymb}      
\usepackage{amsthm}
\usepackage{pdfsync}       
\usepackage{color}
\usepackage[english]{babel}

\usepackage[T1]{fontenc}

\usepackage{color}
\usepackage{float}
\usepackage{tikz,subfigure}
\usetikzlibrary{arrows}
\tikzset{
v/.style={
  circle, draw, inner sep=2pt, minimum size=6pt, fill=white},
l/.style={
  circle, draw, inner sep=2pt, minimum size=6pt, fill=black}
}

\theoremstyle{plain}
\newtheorem{Theorem}{Theorem}[section]
\newtheorem{Proposition}[Theorem]{Proposition}
\newtheorem{Lemma}[Theorem]{Lemma}
\newtheorem{Corollary}[Theorem]{Corollary}

\theoremstyle{Definition}
\newtheorem{Example}[Theorem]{Example}
\newtheorem{Definition}[Theorem]{Definition}
\newtheorem{Remark}[Theorem]{Remark}

 \newcommand \C {{\mathbb C}}
 \newcommand \A{{\mathcal A}}
  \newcommand \B{{\mathcal B}}
  \newcommand \G{{\mathcal G}}

\DeclareMathOperator{\codim}{codim}
\newcommand \spann  {{\rm span}}
\DeclareMathOperator{\rk}{rk}
\renewcommand{\ell}{l}

\def \X(#1){\{x_1,\dots, x_{#1}\}}

\begin{document}

\title {On the Falk invariant of Shi and Linial arrangements}
\begin{abstract} It is an open question to give a combinatorial interpretation of the \emph{Falk invariant} of a hyperplane arrangement, i.e. the third rank of successive quotients in the lower central series of the fundamental group of the arrangement. In this article, we give a combinatorial formula for this invariant in the case of hyperplane arrangements that are complete lift representation of certain gain graphs. As a corollary, we compute the Falk invariant for the cone of the braid,  Shi, Linial and semiorder arrangements.
\end{abstract}

\author{Weili Guo}
\address{Department of Mathematics, Beijing University of Chemical Technology, 15 Beisanhuan East Road, Chaoyang District, Beijing,100013, China.}
\email{guowl@mail.buct.edu.cn}
\author{Michele Torielli}
\address{Department of Mathematics, Hokkaido University, Kita 10, Nishi 8, Kita-Ku, Sapporo 060-0810, Japan.}
\email{torielli@math.sci.hokudai.ac.jp}

\date{\today}
\maketitle


\section{Introduction}
A \textbf{hyperplane arrangement} $\A = \{H_1, \dots , H_n\}$ in $\C^\ell$ is a finite collection  of hyperplanes, i.e. affine subspaces of dimension $\ell-1$. An arrangement $\A$ is called \textbf{central} if $\bigcap_{i=1}^nH_i\ne\emptyset$. In this paper, we mainly consider central arrangements and assume that all the hyperplanes contain the origin. For a thorough treatment of the theory of hyperplane arrangements and recent developments, see \cite{orlterao}, \cite{palezzato2020freehyperplanearbfield}, \cite{palezzato2020lefschetz} and \cite{palezzato2020localiz}.

One of the main goals in the study of hyperplane arrangements is to decide whether a given invariant is combinatorically determined, and, if so, to express it explicitly in terms of the intersection lattice of the arrangement. This is the reason why the hyperplane arrangements coming as representations of different types of graphs have been intensively studied. In fact, one can read their combinatorics directly from the graph.

Between all invariants one of the most interesting and studied is the complement $M:=\C^\ell\setminus\bigcup_{H\in\A}H$ of the arrangement $\A$. It is known that the cohomology ring $H^*(M)$ is completely determined by $L(\A)$ the lattice of intersection of $\A$. Similarly to this result, there are several conjectures concerning the relationship between $M$ and $L(\A)$.  In order to study such problems, Falk introduced the \textbf{global invariant} of the Orlik-Solomon algebra in \cite{falk1990algebra}. The multiplicative invariant,  denoted by $\phi_3$, is now known as the ($3^{rd}$) \textbf{Falk invariant}. In \cite{falk2001combinatorial}, Falk posed as an open problem to give a combinatorial interpretation of $\phi_3$. 

Several authors already studied this invariant. In \cite{schenck2002lower}, Schenck and Suciu studied the lower central series of arrangements and described a formula for the Falk invariant in the case of graphic arrangements. In \cite{guo2017global}, the authors gave a formula for $\phi_3$ in the case of simple signed graphic arrangements. In \cite{guo2017falkinvar}, the authors extended the previous result for signed graphic arrangements coming from graphs without loops. In \cite{guo2017falk}, we described a combinatorial formula for the Falk invariant of a signed graphic arrangement that do not have a $B_2$ as sub-arrangement. In \cite{guoMT2017falk}, we gave a  formula for the Falk invariant $\phi_3$ of the arrangements that are canonical frame representations of gain graphs that do not have a subgraph isomorphic to $B_2$, or loops adjacent to a $\theta$-graph with only three edges and with at most triple parallel edges. 

In this paper, we are interested in the class of hyperplane arrangements that are canonical complete lift representations of a biased graph. Specifically, we will describe a combinatorial formula for the Falk invariant $\phi_3$ for $\mathcal{A}(\G)$, the canonical complete lift representation of a biased graph $\G$ without loops in which there are at most double parallel edges.
This formula will be obtained by counting special type of subgraphs. Finally, we will describe a specialization of the previous formula for the invariant $\phi_3$ of the cone of the braid, Shi, Linial and semiorder arrangements. In all four cases, the formula will just depend on the dimension of the ambient space.

All the computations in this article have been performed using the computer algebra software CoCoA, see \cite{palezzato2018hyperplane}.

\section{Orlik-Solomon algebras of hyperplane arrangements}
In this section, we recall the definition and basic properties of the Orlik-Solomon algebras of hyperplane arrangements. For more details see \cite{orlterao}. Moreover, we will recall the definition and formula of the Falk invariant described in \cite{falk1990algebra} and \cite{falk2001combinatorial}.

Let $\A=\{H_1, \dots, H_n\}$ be a central arrangement of hyperplanes in $\C^{\ell}$.  
Let $E^1:=\bigoplus_{j=1}^n\C e_j$ be the free module generated by $e_1, e_2, \dots, e_n$, where $e_i$ is a symbol corresponding to the hyperplane $H_i$.
Let $E:=\bigwedge E^1$ be the exterior algebra over $\C$. The algebra $E$ is graded via $E=\bigoplus_{p=0}^nE^p$, where $E^p:=\bigwedge^pE^1$.
The $\C$-module $E^p$ is free and has the distinguished basis consisting of monomials $e_S:=e_{i_1}\wedge\cdots\wedge e_{i_p}$,
where $S=\{{i_1},\dots, {i_p}\}$ is running through all the subsets of $\{1,\dots,n\}$ of cardinality $p$ with $i_1<i_2<\cdots<i_p$.
The graded algebra $E$ is a commutative differential graded algebra (i.e. a graded algebra with an added chain complex structure that respects the algebra structure) with respect to the differential $\partial$ of degree $-1$ uniquely defined
by the conditions $\partial e_i=1$ for all $i=1,\dots, n$ and the graded Leibniz formula. Then for every $S\subseteq\{1,\dots,n\}$ of cardinality $p$, we have
$$\partial e_S=\sum_{j=1}^p(-1)^{j-1}e_{S_j},$$
where $S_j$ is the complement in $S$ to its $j$-th element.

For $S\subseteq\{1,\dots,n\}$, put $\bigcap S:=\bigcap_{i\in S}H_i$. The set of all
intersections $L(\A):=\{\bigcap S\mid S\subseteq \{1,\dots,n\}\}$ is called the \textbf{intersection lattice of $\A$}. We endow $L(\A)$ with a partial order defined by $X\le Y$ if and only if $Y\subseteq X$, for all $X,Y\in L(\A)$. Define a rank function on $L(\A)$ by $\rk(X)=\codim(X)$. 
Moreover, we define $\rk(\A)=\codim(\bigcap_{H\in\mathcal{A}}H)$. Associated to $L(\A)$ we have a function $\mu\colon L(\A)\to\mathbb{Z}$, called the \textbf{M\"obius function} of $L(\A)$, defined by
$$\mu(X)=
\begin{cases}
      1 & \text{for } X=\C^l,\\
      -\sum_{Y<X}\mu(Y) & \text{if } X>\C^l.
\end{cases}$$
The \textbf{Whitney numbers} of $\A$ are defined in terms of the M\"obius function by
$$w_p(\A)=\sum_{X\in L(\A), \rk(X)=p}\mu(X).$$
A subset $S\subseteq\{1,\dots,n\}$ is called \textbf{dependent} if the set of polynomials $\{\alpha_i~|~i\in S\}$, with $H_i=\alpha_i^{-1}(0)$, is linearly dependent.

The \textbf{Orlik--Solomon ideal} of $\A$ is the ideal $I=I(\A)$ of $E$ generated by the set $\{\partial e_S~|~S \text{ dependent }\}$.
The algebra $A:=A^\bullet(\A)=E/I(\A)$ is called the \textbf{Orlik--Solomon algebra} of $\A$.

Clearly $I$ is a homogeneous ideal of $E$ and $I^p=I\cap E^p$ whence $A$ is a graded algebra and we can write $A=\bigoplus_{p\ge 0} A^p$, where $A^p=E^p/I^p$.
The map $\partial$ induces a well-defined differential $\partial\colon A^p(\A)\longrightarrow A^{p -1}(\A)$, for any $p>0$.

\begin{Theorem}[{\cite[Theorem 1.3]{falk2001combinatorial}}]\label{theo:osalgandwhitneynumb} The dimension of $A^p$ is equal to the $p$-th Whitney number $w_p(\A)$.
\end{Theorem}

Let $I_k$ be the ideal of $E$ generated by $\sum_{j\le k}I^j$. We call $I_k$ the
 \textbf{$k$-adic Orlik--Solomon ideal} of $\A$. It is clear that $I_k$ is a graded ideal and $(I_k)^p = E^p\cap I_k$. Write $A_k:= E/I_k$ and
$A_k^p:= E^p/(I_k)^p$ which is called \textbf{$k$-adic Orlik--Solomon algebra} by Falk \cite{falk1990algebra}.


In this setup, it is now easy to define the Falk invariant.
\begin{Definition} Consider the map $d$ defined by
$$d\colon E^1\otimes I^2\to E^3,$$
$$d(a\otimes b)=a\wedge b.$$
Then the \textbf{Falk invariant} is defined as
$$\phi_3:=\dim(\ker(d)).$$
\end{Definition}

In \cite{falk1990algebra} and \cite{falk2001combinatorial}, Falk gave a beautiful formula to compute such invariant. In \cite{falk2001combinatorial}, there is typo in the formula, the correct one is the one described below.
\begin{Theorem}[{\cite[Theorem 4.7]{falk2001combinatorial}}]\label{theo:falkinvar} Let $\A=\{H_1, \dots, H_n\}$ be a central arrangement of hyperplanes in $\C^{\ell}$. Then
\begin{equation}\label{eq:falktheorem1}
\phi_3=2\binom{n+1}{3}-nw_2(\A)+\dim(A^3_2).
\end{equation}
\end{Theorem}
\begin{Remark}\label{rem:falkinvariantreduct} Since $\dim(A^3_2)=\dim((E/I_2)^3)=\dim(E^3)-\dim((I_2)^3)$ and $\dim(E^3)=\binom{n}{3}$, then we obtain
\begin{equation}\label{eq:falktheorem}
\phi_3=2\binom{n+1}{3}-nw_2(\A)+\binom{n}{3}-\dim((I_2)^3).
\end{equation}
\end{Remark}

From \cite{schenck2002lower}, we have $\phi_3$ can also be described from the lower central series of the fundamental group $\pi(M)$ of $M=\C^{\ell}\setminus\bigcup_{H\in\A}H$ the complement of the arrangement. In particular, if we consider the lower central series as a chain of normal subgroups $N_i$, for $k \ge 1$, where $N_1 = \pi(M)$ and $N_{k+1} = [N_k,N_1]$, the subgroup generated by commutators of elements in $N_k$ and $N_1$, then $\phi_3$ is the rank of the finitely generated abelian group $N_3/N_4$.

\section{Gain graphs}
In this section, we recall the definitions and basic properties of gain graphs. Furthermore, we will describe the connection between hyperplane arrangements and gain graphs. See \cite{zaslavsky1989biased}, \cite{zaslavsky1991biased} and \cite{zaslavsky2003biased} for a thorough treatment of the subject. See also \cite{bondy2008graph}, for generalities on graph theory.

\subsection{Gain graphs}

An \textbf{additive rational gain graph} $\G= (G, \varphi)$ consists of an underlying graph $|\G|=G=(\mathcal{V}_G,\mathcal{E}_G)$ and a \textbf{gain map} $\varphi\colon \mathcal{E}_G\to \mathbb Q^+$ from the edges of $G$ into the \textbf{gain group} $\mathbb Q^+$ such that $\varphi(\mathtt{e}^{-1})=-\varphi(\mathtt{e})$, where $\mathtt{e}^{-1}$ means $\mathtt{e}$ with its orientation reversed. Since in this paper we will only consider additive rational gain graphs, we will simply call them gain graphs.

Since $\varphi(\mathtt{e}^{-1})=-\varphi(\mathtt{e})$, then $\varphi(\mathtt{e})$ depends on the orientation of $\mathtt{e}$ but neither orientation is preferred.

A \textbf{subgraph} of $\G$ is a subgraph of the underlying graph $|\G|$ with the same gain map, restricted to the subgraph's edges.
A \textbf{walk} is a chain of vertices and edges,
$$P=(v_0, \mathtt{e}_1, v_1, \mathtt{e}_2, \cdots, \mathtt{e}_{k}, v_k),$$
where $v_i\in\mathcal{V}_G,$ $\mathtt{e}_i\in\mathcal{E}_G$, and $\mathtt{e}_i$ oriented from $v_{i-1}$ to $v_i$ for $i\in\{1,2,\cdots,k\}.$
$P$ is determined by its edge sequence, so it may be written as a word 
$$P=\mathtt{e}_1\mathtt{e}_2\cdots \mathtt{e}_k$$ in the free group $\mathfrak{F}(\mathcal{E}_G)$ generated by $\mathcal{E}_G$. 
A walk is a \textbf{path} if it has no repeated vertices except possibly for $v_k=v_0$ if $k>0$
(then it is closed, otherwise it is open). A \textbf{circle} is the edge set of a closed path.
A \textbf{handcuff} consists of two circles with a single vertex in common, or two disjoint circles and a connecting path, as shown in Figure~\ref{Fig:thetahandcuff}.


A path $P = \mathtt{e}_1\mathtt{e}_2\cdots \mathtt{e}_k$ has gain value $\varphi(P) = \varphi(\mathtt{e}_1) +\varphi(\mathtt{e}_2)+\cdots+\varphi(\mathtt{e}_k)$ under $\varphi$. If $P$ is a circle, its gain depends on the starting point and direction, but whether or not the gain equals the identity element $0$ is independent of the starting point and direction. A circle whose gain value is $0$ is called \textbf{balanced}. It is \textbf{unbalanced} if it is not balanced. We call $\G$ \textbf{balanced} if all its circles are balanced, and \textbf{contrabalanced} if it contains no balanced circles at all.  The set of balanced circles of $\G$ is denoted by $\B(\G)$. We call $\left<\G\right>=(G,\B(\G))$ the \textbf{biased graph} associated to $\G$. 

Two biased graphs graphs $\left<\G_1\right>=(G_1,\mathcal B_1)$ and $\left<\G_2\right>=(G_2,\mathcal B_2)$ are \textbf{isomorphic}, written $\left<\G_1\right>\cong \left<\G_2\right>$, if the two underlying graphs are isomorphic, and a circle is in $\mathcal B_1$ if and only if its image is in $\mathcal B_2$.

\begin{figure}[h]
\centering
\begin{tikzpicture}[baseline=10pt]
\draw (1.7,3) node[v,label=above:{$v_1$}](1){};
\draw (0,0) node[v,label=left:{$v_2$}](2){};
\draw (3.4,0) node[v,label=right:{$v_3$}](3){};
\draw[] (1)--(2);
\draw[] (1)--(3);
\draw[] (2)--(3);
\draw[bend right,<<-] (1) to (2);
\draw[bend left,<<-] (1) to (3);
\draw (0.1,1.8) node {1};
\draw (1,1.3) node {0};
\draw (2.4,1.3) node {0};
\draw (1.7,0.2) node {0};
\draw (3.2,1.8) node {1};
\end{tikzpicture}
\caption{Example of gain graph, where the numbers are the gains in the direction indicated by the arrows. The associated biased graph is called $G_\circ$.}\label{Fig:exgaingraph}
\end{figure}
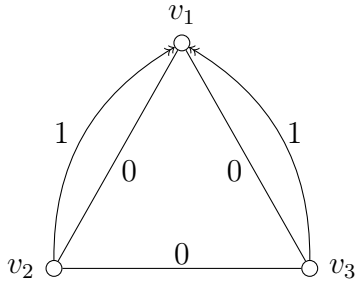

\begin{Example}\label{ex:gaingraphexamp} In Figure \ref{Fig:exgaingraph}, we see a gain graph $\G$ of order $n=3$ with gains in $\mathbb{Q^+}$, the additive group of rational numbers. We adopt the simplified notation $\mathtt e_{ij}(g)$ for an edge $\mathtt e_{ij}$ oriented from $v_i$ to $v_j$ with gain  $\varphi(\mathtt e_{ij})=g$. Then for instance $\mathtt e_{12}(-1)=\mathtt e_{21}(1)$. The circles $C_1:=\{\mathtt e_{12}(0)\mathtt e_{23}(0)\mathtt e_{13}(0)\}$ and $C_2:=\{\mathtt e_{12}(-1)\mathtt e_{23}(0)\mathtt e_{31}(1)\}$ are both balanced. In fact their gains are $\varphi(C_1)=0+0+0=0$ and $\varphi(C_2)=-1+0+1=0$. Therefore $\left<\G\right>=(G,\{C_1,C_2\})$.
\end{Example}

\begin{Theorem}[{\cite[Theorem 3.1]{zaslavsky1991biased}}]\label{theo:completeliftmatroid}
Let $\G$ be a gain graph. Then there is a matroid $L_0(\G)$ whose points are the edges of $\G$ together with an extra point $e_0$ and whose circuits consists of the edge sets of all balanced circles along with all contrabalanced $\theta$-graphs, all unbalanced handcuffs, and all the unions of $e_0$ and an unbalanced circle. 
\end{Theorem}

\begin{figure}[h]
\centering
\subfigure[]
{
\begin{tikzpicture}[baseline=0]
\foreach \x in {0,...,4}
\draw (90+72*\x:0.8) node[v](\x){};
\draw (0)--(1);
\draw (1)--(2);
\draw (2)--(3)--(4);
\draw (4)--(0);
\draw (1)--(4);
\end{tikzpicture}
}
\hspace{14mm}
\subfigure[]
{
\begin{tikzpicture}[baseline=0]
\foreach \x in {0,...,4}
\draw (180+72*\x:0.8) node[v](\x){};
\draw (-1.8,0) node[v](5){};
\draw (-2.8,0) node[v](6){};
\draw (-3.5,0.7) node[v](7){};
\draw (-3.5,-0.7) node[v](8){};
\draw (0)--(1)--(2);
\draw (2)--(3);
\draw (3)--(4);
\draw (4)--(0);
\draw (0)--(5);
\draw (5)--(6);
\draw (7)--(6);
\draw (7)--(8);
\draw (8)--(6);
\end{tikzpicture}
}
\hspace{14mm}
\subfigure[]
{
\begin{tikzpicture}[baseline=0]
\draw (0,0) node[v](1){};
\draw (-0.71, 0.71) node[v](2){};
\draw (-0.71,-0.71) node[v](3){};
\draw ( 0.71, 0.71) node[v](4){};
\draw ( 1.42, 0) node[v](5){};
\draw ( 0.71,-0.71) node[v](6){};
\draw (2)--(1)--(3);
\draw (2)--(3);
\draw (4)--(1)--(6)--(5);
\draw (4)--(5);
\end{tikzpicture}
}
\caption{Examples of (a) a $\theta$-graph, (b) a loose handcuff, (c) a tight handcuff.}\label{Fig:thetahandcuff}
\end{figure}
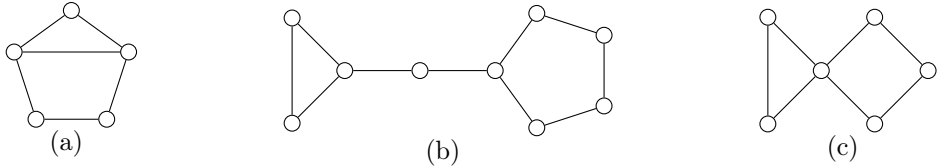

\begin{Definition} Let $\G$ be a gain graph. Then the matroid $L_0(\G)$ is called the \textbf{complete lift matroid} associated to $\G$.
\end{Definition}

Let $\lambda\colon \mathcal{V}_G\to \mathbb{Q}^+$ be any function. Switching $\G$ by $\lambda$ means replacing $\varphi(\mathtt e)$ by $\varphi^\lambda(\mathtt e) :=-\lambda(v)+\varphi(\mathtt e)+\lambda(w)$, where $\mathtt e$ is oriented from $v$ to $w$. The switched graph, $\G^\lambda = (G, \varphi^\lambda)$ is called \textbf{switching equivalent} to $\G$. In general, we will denote by $\G^\ast$ any gain graph that is switching
equivalent to $\G$ for some $\lambda$, and by $[\G]$ the equivalence class of $\G$ under switching equivalance. In \cite{zaslavsky1989biased}, Zaslavsky showed that $\left<\G^\ast\right>=\left<\G\right>$, and that $\G=(G,\varphi)$ is a balanced graph if and only if $\varphi$ switches to the identity gain.


Directly from Theorem \ref{theo:completeliftmatroid}, we have the following result.
\begin{Proposition}\label{prop:biaseqsamematroid}
If $\G_1$ and $\G_2$ are two gain graphs such that $\langle\G_1\rangle=\langle\G_2\rangle$, then $L_0(\G_1)=L_0(\G_2)$.
\end{Proposition}

By the previous proposition and Theorem \ref{theo:completeliftmatroid}, we have the following
\begin{Corollary}\label{corol:switchinggivesamematroid} $L_0(\G^\ast)=L_0(\G)$.
\end{Corollary}

\subsection{Hyperplane arrangement realizations of gain graphs}

In this subsection, we consider $\G= (G, \varphi)$ a gain graph with additive gain group $\mathbb Q^+$, and $\mathcal{V}_G=\{1,\dots,\ell\}$. Moreover, we will assume that in the graph $\G$ there are no loops and that all $2$-circles of $\langle\G\rangle$ are unbalanced. 
\begin{Definition} Let $K$ be a field of characteristic $0$ and $\A(\G)$ be the hyperplane arrangement in $K^{\ell+1}$ consisting of the following hyperplanes
$$\{x_0=0\}\cup\{x_i-x_j+\varphi(\mathtt e_{ij})x_0=0\}\ \text{ for } \mathtt e_{ij}\in \mathcal{E}_G.$$
We will call $\A(\G)$ the \textbf{canonical complete lift representation} of $\G$. 
\end{Definition}
\begin{Example} Consider the gain graph described in Example \ref{ex:gaingraphexamp}. Then we can consider the hyperplane arrangement $\A(\G)\subseteq\mathbb{C}^{4}$ with defining equation $Q=x_0(x_1-x_2-x_0)(x_1-x_2)(x_2-x_3)(x_1-x_3)(x_1-x_3-x_0)$.
\end{Example}
Given a gain graph $\G$, we can now associate to it two matroids: the canonical lift matroid and the matroid of intersections of $\A(\G)$, see \cite{orlterao}, \cite{suyama2019}, \cite{tortsujie2020} and \cite{tsujie2020} for more details. In \cite{zaslavsky2003biased}, Zaslavsky proved that these two matroids coincide. In particular, he proved the following
\begin{Theorem}\label{theo:isolatticeliftarr} $L_0(\G)\cong M(\A(\G))$, where $M(\A(\G))$ is the matroid associated with $\A(\G)$. 
\end{Theorem}

\begin{Proposition}\label{prop:biaseqsamefalkinv}
Let $\G_1$ and $\G_2$ be two gain graphs such that $\langle\G_1\rangle=\langle\G_2\rangle$. Then $\phi_3(\A(\G_1)) = \phi_3(\A(\G_2))$.
\end{Proposition}
\begin{proof} By Proposition~\ref{prop:biaseqsamematroid} and Theorem~\ref{theo:isolatticeliftarr}, $M(\A(\G_1))\cong M(\A(\G_2))$. This implies that $\A(\G_1)$ and $\A(\G_2)$ have isomorphic Orlik--Solomon algebra, and hence they have the same Falk invariant $\phi_3$.
 \end{proof}

As in the case of signed graph (see Corollary 3.11 in \cite{guo2017falk}), by Proposition~\ref{prop:biaseqsamefalkinv} we have the following
\begin{Corollary}\label{corol:switchingsamefalkinv} Let $\G_1$ and $\G_2$ be two gain graphs. If $\G_1$ and $\G_2$ are switching equivalent, then $\phi_3(\A(\G_1)) = \phi_3(\A(\G_2))$.
\end{Corollary}

\section{List of distinguished biased graphs}\label{sect:listgaingraphimportant}
In this section, we list all the gain graphs that we need to describe our main result. Since we consider gain graphs with additive gain group $\mathbb{Q}^+$, 
we describe the underlying graph, together with the list of balanced circles. Since we want to describe the canonical complete lift representation of the given gain graph, we denote by $\mathtt e_0$ the extra point.  

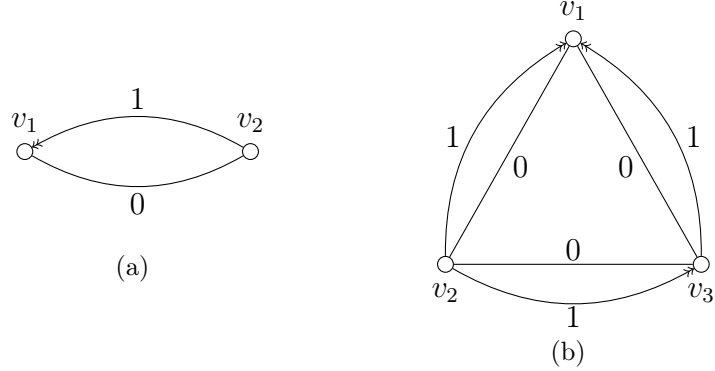
\begin{figure}[h!]
\centering
\subfigure[]
{
\begin{tikzpicture}[baseline=10pt]
\draw (0,1.5) node[v, label=above:{$v_1$}](1){};
\draw (3,1.5) node[v, label=above:{$v_2$}](2){};
\draw[bend left,<<-] (1) to (2);
\draw[bend right] (1) to (2);
\draw (1.5,2.2) node {1};
\draw (1.5, 0.8) node {0};
\end{tikzpicture}
}
\hspace{14mm}
\subfigure[]
{
\begin{tikzpicture}[baseline=10pt]
\draw (1.7,3) node[v, label=above:{$v_1$}](1){};
\draw (0,0) node[v, label=below:{$v_2$}](2){};
\draw (3.4,0) node[v, label=below:{$v_3$}](3){};
\draw[] (1)--(2);
\draw[] (1)--(3);
\draw[] (2)--(3);
\draw[bend right,<<-] (1) to (2);
\draw[bend left,<<-] (1) to (3);
\draw[bend right,->>] (2) to (3);
\draw (0.1,1.7) node {1};
\draw (1,1.3) node {0};
\draw (2.4,1.3) node {0};
\draw (3.3,1.7) node {1};
\draw (1.7,0.2) node {0};
\draw (1.7,-0.7) node {1};
\end{tikzpicture}
}
\caption{The gain graphs $D_2$ and $S_3$.}\label{Fig:disgraphs2}
\end{figure}
\begin{itemize}
\item The biased graph $K_3$ has as underlying graph the complete simple graph on three vertices having the only $3$-circle as balanced circle.
\item The biased graph $K_4$ has as underlying graph the complete simple graph on four vertices and it is balanced.
\item The biased graph $D_2$ has as underlying graph the one depicted in Figure \ref{Fig:disgraphs2} (a) and it is contrabalanced.
\item The biased graph $S_3$ is the one associated with the gain graph depicted in Figure \ref{Fig:disgraphs2} (b), where the gain group is $\mathbb{Q}^+$.
\item The biased graph $G_\circ$ is the one  associated with the gain graph depicted in Figure \ref{Fig:exgaingraph}, where the gain group is $\mathbb{Q}^+$, and it is described in Example~\ref{ex:gaingraphexamp}.
\end{itemize}
\begin{Remark}
The complete lift matroids $L_0(K_3)$, $L_0(K_4)$, $L_0(D_2)$, $L_0(S_3)$ and $L_0(G_{\circ})$ are pairwise distinct and non-isomorphic. 
\end{Remark}

\section{Main theorem}
In this section, we describe how to compute the Falk invariant $\phi_3$ for the canonical complete lift representation $\mathcal{A}(\G)$ of a gain graph $\G$ in which there are no loops and there are at most double parallel edges. Notice that this condition is equivalent to excluding the $4$-point line $U_{2,4}$ as a submatroid of $L_0(\G)\cong M(\A(\G))$. 

In the remaining of the paper, to fix the notation, we suppose that $\G$ is a gain graph whose underlying graph $G$ is on $\ell$ vertices having $n$ edges, and we label only the hyperplanes in $\A(\G)$ as elements of $[n]^+:=\{0,1,\dots, n\}$, where $0$ labels the hyperplane corresponding to the extra point. 

We define the numbers of some subgraphs of a graph $\langle\G\rangle$ as follows:
\begin{itemize}
\item[] $k_l$ denotes the number of subgraphs of $\langle\G\rangle$ isomorphic to a $K_l$,
\item[] $d_2$ denotes the number of subgraphs of $\langle\G\rangle$ isomorphic to a $D_2$,
\item[] $g_\circ$ denotes the number of subgraphs of $\langle\G\rangle$ isomorphic to a $G_\circ$ but not contained in $\langle{S_3}\rangle$,
\item[] $s_3$ denotes the number of subgraphs of $\langle\G\rangle$ isomorphic to $S_3$.
\end{itemize}

The goal of this section is to prove the following theorem.
\begin{Theorem}\label{theorem:ourmain} Let $\G$ be an additive rational gain graph in which there are no loops and there are at most double parallel edges. For an arrangement associated to the gain graph $\G$ via its canonical complete lift representation, we have
\begin{equation}\label{eq:ourmainformular5}
\phi_3=2(k_3+k_4+d_2+g_\circ)+5s_3.
\end{equation}
\end{Theorem}

To prove this theorem we will use Theorem~\ref{theo:falkinvar}. In order to achieve this, firstly we need to identify the triples $S$ in $[n]^+$ that are dependent for the arrangement $\A(\G).$ Then we have the following
\begin{Lemma} Assume $S=(i_1, i_2, i_3)$ with $0\le i_1<i_2<i_3\le n$. Then $S$ is dependent if and only if $i_1, i_2, i_3$ correspond to the edges of a subgraph of $\langle\G\rangle$ that is isomorphic to a $K_3$, or $i_1=0$ and $i_2, i_3$ correspond to the edges of a subgraph of $\langle\G\rangle$ that is isomorphic to a $D_2$.
\end{Lemma}

Since a dependent triple $S$ corresponds to a circuit of size $3$ in $M(\A(\G))$, we call such S a \textbf{$3$-circuit}. Moreover, we will write
$$\mathcal{C}_3:=\spann\{e_S\in E~|~S \text{ is a $3$-circuit}\}$$
which is a subset of $E$ as a vector space over $\C$.

\begin{Remark} Notice that the $3$-circuits are exactly the balanced $3$-circles and a $D_2$ with the extra point $\mathtt e_0$. If $\G_1$ and $\G_2$ are two switching equivalent gain graphs, then $\mathcal{C}_3(\G_1)=\mathcal{C}_3(\G_2)$.
\end{Remark}

Since $e_{ijk}=-e_{jik}$, it is clear that the dimension of the vector space $\mathcal{C}_3$ is $k_3+d_2$. Let $C'_3$ be a basis of $\mathcal{C}_3$ consisting of elements corresponding to the subgraphs of $\langle\G\rangle$ isomorphic to a $ K_3$, or a $D_2$ with the extra point $\mathtt e_0$. 


Under the assumption that $\G$ has no loops and there are at most double parallel edges, directly from the definition of Whitney numbers, we have the following
\begin{Lemma}\label{lem:dimA2} $w_2(\A(\G))=\binom{n+1}{2}-k_3-d_2$.
\end{Lemma}
Using Theorem \ref{theo:falkinvar} and Remark \ref{rem:falkinvariantreduct}, to prove Theorem \ref{theorem:ourmain}, we just need to describe $\dim((I_2)^3)$. To do so, consider 
$$C_3:=\{e_t\partial e_{ijk}~|~e_{ijk}\in C'_3,t\in\{i,j,k\}\},$$
and
$$F_3:=\{e_t\partial e_{ijk}~|~e_{ijk}\in C'_3,t\in[n]^+\setminus\{i,j,k\}\}.$$

By construction $(I_2)^3=I^2\cdot E^1=\spann\{e_t\partial e_{ijk}~|~e_{ijk}\in C'_3,t\in[n]\}$, and hence
$$(I_2)^3=\spann(C_3)+\spann(F_3).$$

\begin{Lemma} Let $\G$ be a gain graph in which there are no loops and there are at most double parallel edges. For an arrangement associated to the gain graph $\G$ via its canonical complete lift representation , we have
$$I^3_2=\spann(C_3)\oplus\spann(F_3).$$
\end{Lemma}
\begin{proof} Since $\G$ does not contain any loops and two distinct vertices are connected by at most two edges, any two $3$-circuits share at most one element. This then gives us that $\spann(C_3)\cap\spann(F_3)=\{0\}$.
 \end{proof}
Next we have
\begin{align*}
\dim(I^3_2) & =\dim(\spann(C_3))+\dim(\spann(F_3))\\
& =k_3+d_2+\dim(\spann(F_3)).
\end{align*}
Hence, to prove our main result we need to be able to compute $\dim(\spann(F_3))$. To do so, consider the following sets
\begin{align*}
F^1_3 & :=\{e_t\partial e_{ijk}\in F_3~|~t,i,j,k \text{ do not correspond to edges of the same } K_4, S_3, G_\circ\},\\
F^2_3 & :=\{e_t\partial e_{ijk}\in F_3~|~t,i,j,k \text{ correspond to edges of the same } K_4\}, \\
F^3_3 & :=\{e_t\partial e_{ijk}\in F_3~|~t,i,j,k \text{ correspond to edges of the same } G_\circ\text{ not in an } S_3\},\\
F^4_3 & :=\{e_t\partial e_{ijk}\in F_3~|~t,i,j,k \text{ correspond to edges of the same } S_3\}.\\
\end{align*}
Notice that in the previous four sets, any of $i, j, k, t$ may be $0$.

For a pair of parallel edges $(i,j)$ that form an unbalanced circle, we will consider the $3$-circuit $(0,i,j).$ Hence, $e_0$ will also appear in $F^3_3$ and $F^4_3.$

\begin{Lemma} For an arrangement associated to a gain graph $\G$ via its canonical complete lift representation in which there are no loops and there are at most double parallel edges, we have
$$\spann(F_3)=\spann(F^1_3)\oplus(\spann(\bigcup_{i=2}^4F^i_3)).$$
Moreover, $\spann(F^3_3)\cap\spann(F^4_3)=\{0\}$.
\end{Lemma}

\begin{proof} For any element $e_t\partial e_{ijk}$ of $F^1_3$, we assert that at least one of its terms $e_{tjk}, $ $ e_{tik}, $ $e_{tij}$ appears only in the expression of $e_t\partial e_{ijk}\in F^1_3$ and not in the expression of any other element in $F^2_3\cup F_3^3\cup F^4_3$. So $e_t\partial e_{ijk}$ can not be expressed linearly by the elements of $F^2_3, F_3^3, F^4_3$.

Since the edges $t, i, j, k$ are not in the same $ K_4,  G_\circ, S_3$, and we do not consider the graphs in which there are  loops or triple parallel edges, we should only consider three cases about the edge $t$: it can be adjacent to none of the edges $i,j,k$, to two of them, or to all of them. Notice that if $t=0$, we say that it is adjacent to $i, j$ if $i, j$ are parallel edges forming an unbalanced circle in $\G$.

Assume that the edge $t$ is adjacent to none of the edges $i,j,k$. This implies that $t$ and none of $i,j,k$ can appear in the same $3$-circuit. Hence any term of $e_t\partial e_{ijk}$ of $F^1_3$ will not appear in any of $F^2_3, F_3^3, F^4_3$.

Assume now that the edge $t$ is adjacent to two of the edges $i,j,k$, then we should consider several possibilities.
Suppose that $0\notin\{t,i,j,k\}.$ If all the terms of the element $e_t\partial e_{ijk}\in F^1_3$ appear in $F^2_3$, $F_3^3,$ or $F^4_3$, then $t,i,j,k$ have to appear in the same $ K_4$, but this is impossible by construction.
Suppose that $t=0$. If all the terms of the element $e_t\partial e_{ijk}\in F^1_3$ appear in $F^2_3, F_3^3, F^4_3$, then $t,i,j,k$ have to appear in the same $G_\circ$ or in the same $ S_3$, but this is impossible by construction.
Suppose that $t\ne0$ and $0\in\{i,j,k\}$. In this case two edges in $i,j,k$ are the edges of a $ D_2$. If all the terms of the element $e_t\partial e_{ijk}\in F^1_3$ appear in $F^2_3$, $F_3^3,$ or $F^4_3$, then, in this case, $t,i,j,k$ have to appear in the same $ G_\circ$ or in the same $S_3$, but this is impossible by construction.

Finally, assume that the edge $t$ is adjacent to all the edges $i,j,k$. Since the underlying graph has at most double parallel edges and no loops among its edges, then $t\ne0$ and we should consider only two possibilities.
Suppose that $0\notin\{t,i,j,k\}.$ If all the terms of the element $e_t\partial e_{ijk}\in F^1_3$ appear in $F^2_3$, $F_3^3,$ or $F^4_3$, then $t,i,j,k$ have to appear in the same $ G_\circ$ or in the same $ S_3$, but this is impossible by construction.
Suppose that $0\in\{i,j,k\}$. In this case two edges in $\{i,j,k\}$ are the edges of a $ D_2$. 
 If all the terms of the element $e_t\partial e_{ijk}\in F^1_3$ appear in $F^2_3$, $F_3^3,$ or $F^4_3$, then, in this case, $t,i,j,k$ have to appear in the same $ G_\circ$ or in the same $ S_3$, but this is impossible by construction.

Therefore, for any element $e_t\partial e_{ijk}\in F^1_3$, at least one of the terms $e_{tjk}, e_{tik}, e_{tij}$ appears only in the expression of $e_t\partial e_{ijk}$. This shows that $$\spann(F^1_3)\bigcap(\spann(\bigcup_{i=2}^4F^i_3))=\{0\}.$$

 Since clearly
$\spann(F_3)=\sum_{i=1}^4\spann(F^i_3),$
this concludes the proof of the first part of the statement.

Since elements in $F^3_3$ and $F^4_3$ share at most three indices, and if it is exactly three, one of them is $0$, this implies that no elements of $F^3_3$ can be written as a linear combination of elements of $F^4_3$. Vice versa, no elements of $F^4_3$ can be written as a linear combination of elements of $F^3_3$, and hence, we have the second part of the statement.
 \end{proof}

Differently from the situation discussed in \cite{guoMT2017falk}, in this setting it might happen that $\spann(F^2_3)\cap\spann(F^3_3)\ne\{0\}$ or $\spann(F^2_3)\cap\spann(F^4_3)\ne\{0\}$. In particular, we have the following straightforward result.

\begin{Lemma}\label{lem:nondirectsumm} In the graph $\langle\G\rangle$ there is a subgraph isomorphic to a $K_3$ with edges $i,j,k$ contained in a subgraph isomorphic to a $G_\circ$ and one isomorphic to a $K_4$ at the same time if and only if $\{0\}\ne\spann(F^2_3)\cap\spann(F^3_3)\supseteq\spann(e_0 \partial e_{ijk})$. Moreover, the elements of the type $e_0 \partial e_{ijk}$ generate the intersection.

Similarly, in the graph $\langle\G\rangle$ there is a subgraph isomorphic to a $K_3$ with edges $i,j,k$ contained in a subgraph isomorphic to a $S_3$ and one isomorphic to a $K_4$ at the same time if and only if $\{0\}\ne\spann(F^2_3)\cap\spann(F^4_3)\supseteq\spann(e_0 \partial e_{ijk})$. Moreover, the elements of the type $e_0 \partial e_{ijk}$ generate the intersection.

However, in both cases, $e_{0jk},e_{0ik},e_{0ij}$ only appear one time in $F_3^2$.
\end{Lemma}
\begin{proof} If in the graph $\langle\G\rangle$ there is a subgraph isomorphic to a $K_3$ with edges $i,j,k$ contained in a subgraph isomorphic to a $G_\circ$ and one isomorphic to a $K_4$ at the same time, then clearly $e_0 \partial e_{ijk}\in \spann(F^2_3)\cap\spann(F^3_3)$.
On the other hand, if $e_0 \partial e_{ijk}\in \spann(F^2_3)\cap\spann(F^3_3)$, then $i,j,k$ are edges that belongs at the same time to a subgraph isomorphic to a $G_\circ$ and one isomorphic to a $K_4$. Hence, they are the edges of subgraph isomorphic to a $K_3$.
In addition, if $\{0\}\ne\spann(F^2_3)\cap\spann(F^3_3)$, then this intersection is clearly spanned by elements of the type $e_0 \partial e_{ijk}$.

One the other hand, if in the graph $\langle\G\rangle$ there is a subgraph isomorphic to a $K_3$ with edges $i,j,k$ contained in a subgraph isomorphic to a $S_3$ and one isomorphic to a $K_4$ at the same time, then clearly $e_0 \partial e_{ijk}\in \spann(F^2_3)\cap\spann(F^4_3)$.
On the other hand, if $e_0 \partial e_{ijk}\in \spann(F^2_3)\cap\spann(F^4_3)$, then $i,j,k$ are edges that belongs at the same time to a subgraph isomorphic to a $S_3$ and one isomorphic to a $K_4$. Hence, they are the edges of subgraph isomorphic to a $K_3$.
In addition, if $\{0\}\ne\spann(F^2_3)\cap\spann(F^4_3)$, then this intersection is clearly spanned by elements of the type $e_0 \partial e_{ijk}$.
 \end{proof}

\begin{Example}\label{ex:G0comput} We consider the dimension of $\spann(F_3)$ for the arrangement $\A(G_\circ)$ associated to the graph $G_\circ$ (see Figure \ref{Fig:exgaingraph}). We label the hyperplanes in $\A(G_\circ)$ corresponding to the edges $\mathtt e_{21}(1)$, $\mathtt e_{12}(0)$, $\mathtt e_{13}(0)$, $\mathtt e_{31}(1)$, $\mathtt e_{23}(0)$ as $1, 2, 3, 4, 5$, and label the hyperplane corresponding to the extra point $\mathtt e_0$ as $0$.
In the matroid $M(\A(G_\circ))$ we have as $3$-circuits $S:=\{012,034,235,145\}$. 
Then the number of the elements in $F_3$ is $12$, listed as follows.
$$ e_3 \partial e_{012}=e_{013}-e_{023}+e_{123}, \qquad e_4 \partial e_{012}=e_{014}-e_{024}+e_{124},$$
$$e_5 \partial e_{012}=e_{015}-e_{025}+e_{125}, \qquad  e_1 \partial e_{034}=e_{134}-e_{013}+e_{014},$$
$$e_2 \partial e_{034}=e_{234}-e_{023}+e_{024}, \qquad  e_5 \partial e_{034}=e_{345}-e_{045}+e_{035},$$
$$e_0 \partial e_{235}=e_{035}-e_{025}+e_{023}, \qquad  e_1 \partial e_{235}=e_{135}-e_{125}-e_{123},$$
$$e_4 \partial e_{235}=e_{245}-e_{345}+e_{234}, \qquad  e_0 \partial e_{145}=e_{045}-e_{015}+e_{014},$$
$$e_2 \partial e_{145}=e_{245}-e_{124}+e_{125}, \qquad  e_3 \partial e_{145}=e_{345}-e_{134}+e_{135}.$$
Then an easy computation shows that in this case $\dim(\spann(F_3))=10$. 
\end{Example}
\begin{Example}\label{ex:G_1comput} We consider the dimension of $\spann(F_3)$ for the arrangement $\A(S_3)$ associated to the gain graph $S_3$ (see Figure \ref{Fig:disgraphs2}(b)).
In this situation we have six $3$-circuits in $M(\A(S_3))$. 
Then the number of the elements in $F_3$ is $24$, and they are all the elements of the form
$$e_t\partial e_{ijk}=e_{tjk}-e_{tik}+e_{tij},$$
where $(i,j,k)$ is a $3$-circuit and $t\notin\{i,j,k\}$.
Then an easy computation shows that in this case $\dim(\spann(F_3))=19$. 
\end{Example}

\begin{Example}\label{ex:dimf3all} Similarly to the previous examples, we can compute $\dim(\spann(F_3))$ directly for the gain graph $K_4$, and show $\dim(\spann(F_3))=14.$  
\end{Example}
\begin{Lemma}\label{lem:dimg0s3} With the previous notations, $\dim(\spann(F_3^3))=10g_\circ$, $\dim(\spann(F_3^4))=19s_3$.
\end{Lemma}
\begin{proof} Assume that in the graph $\langle\G\rangle$ there are exactly $p$ distinct subgraphs isomorphic to a $G_\circ$, $\langle\G_1\rangle,\dots, \langle\G_p\rangle$, none of which is a subgraph of a graph isomorphic to $ S_3$.  Consider
$$F^3_{3,r}:=\{e_t\partial e_{ijk}~|~e_{ijk}\in C'_3,t\in[n]^+\setminus\{i,j,k\}, i,j,k \in \\\langle \G_r\rangle\}.$$
Since three edges in the underlying graph of $\G$ can not appear in two distinct $G_\circ$ at the same time, then none of the terms of the element $e_t\partial e_{ijk}\in F^3_{3,r}$ appear in the elements of $F_3^3\setminus F^3_{3,r}$. This shows that
$$\spann(F^3_3)=\bigoplus_{r=1}^p \spann(F^3_{3,r}).$$
By Example \ref{ex:G0comput}, $\dim(\spann(F^3_{3,r}))=10$ for all $r=1,\dots, p.$ This then implies that
$$\dim(\spann(F^3_3))=\sum_{r=1}^p \dim(\spann(F^3_{3,r}))=10g_\circ.$$

Using Example \ref{ex:G_1comput}, the same exact argument used in this case will prove that $\dim(\spann(F_3^4))=19s_3$.
 \end{proof}

Notice that the argument of the previous lemma cannot be utilized to compute $\dim(\spann(F_3^2))$. This is because if in the graph $\langle\G\rangle$ there is a subgraph isomorphic to a $K_3$ with edges $i,j,k$ contained in two distinct subgraphs isomorphic to a $K_4$ at the same time, $\G_1, \G_2$, then $\{0\}\ne\spann(F^2_{3,1})\cap\spann(F^2_{3,2})\supseteq\spann(e_0 \partial e_{ijk})$. Moreover, the elements of the type $e_0 \partial e_{ijk}$ generate the intersection. This fact together with a similar argument to the one in the proof of Lemma \ref{lem:dimg0s3} gives us the following result.
\begin{Lemma}\label{lem:dimK4tot} With the previous notations, 
\begin{equation}\label{eq:dimK4F3}\dim(\spann(F_3^2))=14k_4-\sum_{i\ge2}(i-1)\lambda_i,\end{equation}
where $\lambda_i$ is the number of subgraphs of $\langle\G\rangle$ isomorphic to a $K_3$ contained in exactly $i$ distinct subgraphs of $\langle\G\rangle$ isomorphic to a $K_4$ at the same time.
\end{Lemma}
Notice that, since we are dealing with finite graphs, the sum in the formula \eqref{eq:dimK4F3} is a finite sum.

\begin{Lemma}\label{lemm:dimI32} For an arrangement associated to a gain graph $\G$ via its canonical complete lift representation in which there are no loops and there are at most double parallel edges, we have
$$\dim(I^3_2)=(n-1)(k_3+d_2)-2k_4-2g_\circ-5s_3.$$
\end{Lemma}
\begin{proof} To prove the statement, we need to compute  $\dim(\spann(F_3))$. From Lemma \ref{lem:nondirectsumm}, let $\gamma$ be the number of subgraphs of $\langle\G\rangle$ isomorphic to a $K_3$ contained in a subgraph isomorphic to a  $G_\circ$ and one isomorphic to a $K_4$ or in a subgraph isomorphic to a $S_3$ and one isomorphic to a $K_4$ at the same time. From Lemma \ref{lem:dimK4tot}, let $\lambda:=\sum_{i\ge2}(i-1)\lambda_i$.
By the previous lemmas
$$\dim(\spann(F_3))=\dim(\spann(F^1_3))+\dim(\spann(\bigcup_{i=2}^4F_3^i))=$$
$$=[(n-2)(k_3+d_2)-16k_4+\lambda-12g_\circ-24s_3+\gamma]+\dim(\spann(\bigcup_{i=2}^4 F_3^i)).$$
$$=[(n-2)(k_3+d_2)-16k_4+\lambda-12g_\circ-24s_3+\gamma]+$$
$$ 14k_4-\lambda+10g_\circ+19s_3-\gamma$$
$$ =(n-2)(k_3+d_2)-2k_4-2g_\circ-5s_3.$$
The thesis follows from the equality
$$\dim(I^3_2) = k_3+d_2+\dim(\spann(F_3)).$$
 \end{proof}

\begin{proof}[Proof of Theorem~\ref{theorem:ourmain}]
By Remark \ref{rem:falkinvariantreduct} and Lemma \ref{lem:dimA2} we have
$$\phi_3=2\binom{n+2}{3}-(n+1)\bigg[\binom{n+1}{2}-k_3-d_2\bigg]+\binom{n+1}{3}-\dim(I^3_2).$$
Because $2\binom{n+2}{3}-(n+1)\binom{n+1}{2} +\binom{n+1}{3}=0$, then from Lemma \ref{lemm:dimI32} we obtain
$$\phi_3=2(k_3+k_4+d_2+g_\circ)+5s_3.$$ 

\end{proof}

\begin{figure}[htbp]
\centering
\begin{tikzpicture}[baseline=10pt]
\draw (0,4) node[v, label=above:{$v_1$}](1){};
\draw (0,0) node[v,label=below:{$v_2$}](2){};
\draw (4,0) node[v,label=below:{$v_3$}](3){};
\draw (4,4) node[v,label=above:{$v_4$}](4){};
\draw[] (1)--(2);
\draw[] (1)--(3);
\draw[] (1)--(4);
\draw[] (2)--(3);
\draw[] (2)--(4);
\draw[] (3)--(4);
\draw[bend right,->>] (1) to (2);
\draw[bend left,->>] (1) to (4);
\draw[bend right,->>] (1) to (3);
\draw[bend right,<<-] (3) to (4);
\draw[bend right,<<-] (2) to (3);
\draw (-0.8,2) node {1};
\draw (-0.2,2) node {0};
\draw (4.8,2) node {1};
\draw (4.2,2) node {0};
\draw (2,4.8) node {1};
\draw (2,4.2) node {0};
\draw (2,-0.2) node {0};
\draw (1.4,3) node {0};
\draw (2.6,3) node {0};
\draw (2,1.15) node {1};
\draw (2,-0.8) node{1};
\end{tikzpicture}
\caption{The gain graph $\G$.}
\label{fig:finalexample}
\end{figure}
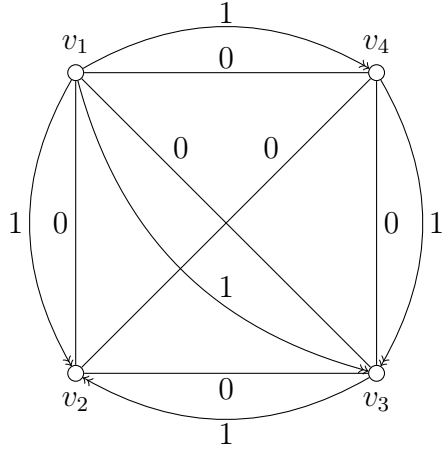

Let us see how our formula works on a non-trivial example.
\begin{Example} We want to compute $\phi_3$ for the arrangement associated to the gain graph $\G$ of Figure \ref{fig:finalexample}.

In order to compute $\phi_3$ with the formula \eqref{eq:ourmainformular5}, we need to compute the following:
\begin{itemize}
\item $k_3=|\big\{\{\mathtt e_{12}(1)\mathtt e_{41}(-1)\mathtt e_{24}(0)\},\{\mathtt e_{12}(1)\mathtt e_{23}(0)\mathtt e_{31}(-1)\},\{\mathtt e_{12}(1)\mathtt e_{23}(-1)\mathtt e_{31}(0)\},$\\$\{\mathtt e_{12}(0)\mathtt e_{23}(0)\mathtt e_{31}(0)\},\{\mathtt e_{12}(0)\mathtt e_{24}(0)\mathtt e_{41}(0)\},\{\mathtt e_{13}(0)\mathtt e_{34}(0)\mathtt e_{41}(0)\},$ \\$\{\mathtt e_{23}(0)\mathtt e_{34}(0)\mathtt e_{43}(0)\},$ $\{\mathtt e_{14}(1)\mathtt e_{43}(0)\mathtt e_{31}(-1)\},$ $\{\mathtt e_{14}(0)\mathtt e_{43}(1)\mathtt e_{31}(-1)\}\big\}|=9;$
\item $d_2=|\big\{\{\mathtt e_{12}(1)\mathtt e_{21}(0)\},\{\mathtt e_{14}(1)\mathtt e_{41}(0)\},\{\mathtt e_{13}(1)\mathtt e_{31}(0)\},\{\mathtt e_{23}(0)\mathtt e_{32}(1)\},$\\$\{\mathtt e_{43}(1)\mathtt e_{34}(0)\}\big\}|=5;$
\item $k_4=|\big\{\{\mathtt e_{12}(0)\mathtt e_{23}(0)\mathtt e_{34}(0)\mathtt e_{41}(0)\mathtt e_{13}(0)\mathtt e_{24}(0)\},$\\$\{\mathtt e_{12}(1)\mathtt e_{14}(1)\mathtt e_{13}(1)\mathtt e_{24}(0)\mathtt e_{23}(0)\mathtt e_{34}(0)\}\big\}|=2;$
\item $g_\circ=|\big\{\{\mathtt e_{12}(1)\mathtt e_{21}(0)\mathtt e_{14}(1)\mathtt e_{41}(0)\mathtt e_{24}(0)\}\big\}|=1.$
\item $s_3=|\big\{\{\mathtt e_{12}(1)\mathtt e_{21}(0)\mathtt e_{13}(1)\mathtt e_{31}(0)\mathtt e_{23}(0)\mathtt e_{32}(1)\},$\\$\{\mathtt e_{14}(1)\mathtt e_{41}(0)\mathtt e_{13}(1)\mathtt e_{31}(0)\mathtt e_{43}(1)\mathtt e_{34}(0)\}\big\}|=2.$
\end{itemize} 
From formula \eqref{eq:ourmainformular5}, we obtain $$\phi_3=2(9+5+2+1)+5\cdot2=44.$$ Notice that if we would try to compute the dimension of $F_3$ directly, we would have to write $126$ equations in the $e_{ijk}$.
\end{Example}

\section{The cone of the braid, Shi, Linial and semiorder arrangements}

In this section, we compute the Falk invariant of the cone of several known arrangements using Theorem \ref{theorem:ourmain}, where the \textbf{cone} of an arrangements is obtained by homogenizing all the defining polynomials with respect to $x_0$ and adding the hyperplane $\{x_0 = 0\}$. Notice that coning allows one to transform any arrangement $\A$ in $K^\ell$ with $n$ hyperplanes into a central arrangement $c\A$ with $n + 1$ hyperplanes in $K^{\ell+1}$, see \cite{orlterao}.

The \textbf{braid arrangement} $\mathcal B_l$ in $\mathbb{C}^\ell$ is the arrangement consisting of the hyperplanes
$$\{x_i-x_j=0\} \text{ for } 1\le i< j\le\ell.$$
It is easy to see that the cone of the braid arrangement $c(\mathcal{B}_l)$ is the canonical complete lift representation of the complete graph $K_l$ on $\ell$ vertices, such that each edge has gain equal to $0$.
\begin{Theorem} The Falk invariant of the cone of the braid arrangement $c(\mathcal{B}_l)$ is given by
$$\phi_3(c(\mathcal B_l))=2\binom{\ell+1}{4}=\frac{{\ell}({\ell}+1)({\ell}-1)({\ell}-2)}{12}.$$
\end{Theorem}
\begin{proof}
In the graph $ K_l$ any three vertices can form a $K_3$ and any four vertices can form a $K_4$, so $k_3=\binom{\ell}{3}$ and $k_4=\binom{\ell}{4}$. Moreover, there are no subgraphs isomorphic to a $D_2$, or a $G_\circ$, or a $S_3$.

From formula \eqref{eq:ourmainformular5}, we obtain
$$\phi_3(c(\mathcal B_l))=2\bigg(\binom{\ell}{3}+\binom{\ell}{4}\bigg)=\frac{{\ell}({\ell}+1)({\ell}-1)({\ell}-2)}{12}.$$
\end{proof}

The braid arrangement has a number of ``deformations'' of considerable interests, see \cite{orlterao} and \cite{stanley2004introduction} for more details. We will just define three of those: the Shi arrangement, the Linial arrangement and the semiorder arrangement. 

The \textbf{Shi arrangement} $\mathcal S_\ell$ in $\mathbb{C}^\ell$ is the arrangement consisting of the hyperplanes
$$\{x_i-x_j=0\}\cup\{x_i-x_j-1=0\} \text{ for } 1\le i< j\le\ell.$$

Notice that the cone of the Shi arrangement $c(\mathcal S_\ell)$ is the arrangement in $\mathbb{C}^{\ell+1}$ consisting of the hyperplanes
$$\{x_0=0\}\cup\{x_i-x_j=0\}\cup\{x_i-x_j-x_0=0\} \text{ for } 1\le i< j\le\ell.$$
This implies that $c(\mathcal S_\ell)$ is the canonical complete lift representation of the gain graph $\G_{\mathcal S}$, where $\G_{\mathcal S}$ is the gain graph with underlying graph $G$ on $\mathcal{V}_G=[\ell]$ such that for any two distinct vertices $i,j\in[\ell]$ with $i<j$, there are exactly two parallel edges $\mathtt e_{ij}$ and $\mathtt e'_{ij}$, with gains respectively $\varphi(\mathtt e_{ij})=0$ and $\varphi(\mathtt e'_{ij})=-1$. 
\begin{Theorem} The Falk invariant of the cone of the Shi arrangement $c(\mathcal S_\ell)$ is given by
$$\phi_3(c(\mathcal S_\ell))=\frac{{\ell}({\ell}-1)(2{\ell}^2+{\ell}-4)}{6}.$$
\end{Theorem}
\begin{proof}
In the graph $\langle\G_{\mathcal S}\rangle$ any two vertices can form a $D_2$, so the number of subgraph isomorphic to a $D_2$ is $\binom{\ell}{2}.$ Any three vertices can form a $S_3$ with $3$ subgraphs isomorphic to a $K_3,$ so the number of subgraphs isomorphic to a $S_3$ is $\binom{\ell}{3}$,
while the number of the subgraphs isomorphic to a $K_3$ is $3\binom{\ell}{3}.$  Moreover, any four vertices gives us $4$ subgraphs isomorphic to a $K_4$, so the number of subraphs isomorphic to a $K_4$ is $4\binom{\ell}{4}.$ Finally. there is no subgraph isomorphic to a $G_\circ.$

From formula \eqref{eq:ourmainformular5}, we obtain
$$\phi_3(c(\mathcal S_\ell))=2\bigg(3\binom{\ell}{3}+\binom{\ell}{2}+4\binom{\ell}{4}+0\bigg)+5\binom{\ell}{3}$$
$$=\frac{{\ell}({\ell}-1)(2{\ell}^2+{\ell}-4)}{6}.$$
 \end{proof}

The \textbf{Linial arrangement} $\mathcal{L}_\ell$ in $\mathbb{C}^\ell$ is the arrangement consisting of the hyperplanes
$$\{x_i-x_j-1=0\} \text{ for } 1\le i< j\le\ell.$$
Notice that the cone of the Linial arrangement $c(\mathcal{L}_\ell)$ is the arrangement in $\mathbb{C}^{\ell+1}$ consisting of the hyperplanes
$$\{x_0=0\}\cup\{x_i-x_j-x_0=0\} \text{ for } 1\le i< j\le\ell.$$
This implies that $c(\mathcal{L}_\ell)$ is the canonical complete lift representation of the gain graph $\G_{\mathcal{L}}$, where $\G_{\mathcal{L}}$ is the gain graph with underlying graph $G=K_\ell$, the complete graph on $\ell$ vertices, such that for any two distinct vertices $i,j\in[\ell]$ with $i<j$, the edge $\mathtt e_{ij}$ has gain equal to $\varphi(\mathtt e_{ij})=-1$. 
\begin{Theorem} The Falk invariant of the cone of the Linial arrangement $c(\mathcal{L}_\ell)$ is zero for any $\ell$.
\end{Theorem}
\begin{proof} In the graph $\langle\G_{\mathcal{L}}\rangle$ there are no subgraphs isomorphic to a $D_2$, or a $K_3$, or a $K_4$, or a $G_\circ$, or a $S_3$. This implies by the formula \eqref{eq:ourmainformular5} that $\phi_3(c(\mathcal{L}_\ell))=0$.
 \end{proof}

The \textbf{semiorder arrangement} $\mathcal C_l^{\circ}$ in $\mathbb{C}^\ell$ is the arrangement consisting of the hyperplanes
$$\{x_i-x_j+1=0\}\cup\{x_i-x_j-1=0\} \text{ for } 1\le i< j\le\ell.$$
Notice that the cone of the semiorder arrangement $c(\mathcal C_l^{\circ})$ is the arrangement in $\mathbb{C}^{\ell+1}$ consisting of the hyperplanes
$$\{x_0=0\}\cup\{x_i-x_j+x_0=0\}\cup\{x_i-x_j-x_0=0\} \text{ for } 1\le i< j\le\ell.$$
This implies that $c(\mathcal C_l^{\circ})$ is the canonical complete lift representation of the gain graph $\G_{\mathcal{S}}$, where $\G_{\mathcal{S}}$ is the gain graph with underlying graph $G$ on $\mathcal{V}_G=[\ell]$ such that for any two distinct vertices $i,j\in[\ell]$ there are exactly two parallel edges, with gains respectively $1$ and $-1$. 
\begin{Theorem} The Falk invariant of the cone of the semiorder arrangement $c(\mathcal C_l^{\circ})$ is given by
$$\phi_3(c(\mathcal C_l^{\circ}))=\ell(\ell-1).$$
\end{Theorem}
\begin{proof}
In the graph $\langle\G_{\mathcal{S}}\rangle$ any two vertices can form a $D_2$, so the number of subgraph isomorphic to a $D_2$ is $\binom{\ell}{2}$. Moreover, there are no subgraphs isomorphic to a $K_3$, or a $K_4$, or a $G_\circ$, or a $S_3$.

From formula \eqref{eq:ourmainformular5}, we obtain
$$\phi_3(c(\mathcal C_l^{\circ}))=2\binom{\ell}{2}=\ell(\ell-1).$$

 \end{proof}

\bigskip
\paragraph{\textbf{Acknowledgements}} During the preparation of this article the second author was supported by JSPS Grant-in-Aid for Early-Career Scientists (19K14493).


\begin{thebibliography}{}

\bibitem{bondy2008graph}
J.~A. Bondy and U.~S.~R. Murty.
\newblock {\em Graph Theory}, volume 244 of Graduate texts in
  mathematics.
\newblock Springer Science and Media, 2008.

\bibitem{falk1990algebra}
M.~Falk.
\newblock On the algebra associated with a geometric lattice.
\newblock {\em Advances in Mathematics}, 80(2):152--163, 1990.

\bibitem{falk2001combinatorial}
M.~Falk.
\newblock Combinatorial and algebraic structure in {O}rlik-{S}olomon algebras.
\newblock {\em European Journal of Combinatorics}, 22(5):687--698, 2001.

\bibitem{guo2017global}
Q.~Guo, W.~Guo, W.~Hu, and G.~Jiang.
\newblock The global invariant of signed graphic hyperplane arrangements.
\newblock {\em Graphs and Combinatorics}, 33:527--535, 2017.

\bibitem{guo2017falkinvar}
W.~Guo, Q.~Guo, and G.~Jiang.
\newblock Falk invariants of signed graphic arrangements.
\newblock {\em Graphs and Combinatorics}, 34:1247--1258, 2018.

\bibitem{guoMT2017falk}
W.~Guo and M.~Torielli.
\newblock On the {F}alk invariant of hyperplane arrangements attached to gain
  graphs.
\newblock  {\em Australasian Journal of Combinatorics}, 77(3): 301--317, 2020.

\bibitem{guo2017falk}
W.~Guo and M.~Torielli.
\newblock On the {F}alk invariant of signed graphic arrangements.
\newblock {\em Graphs and Combinatorics}, 34(3):477--488, 2018.

\bibitem{orlterao}
P.~Orlik and H.~Terao.
\newblock {\em Arrangements of Hyperplanes}, volume 300 of Grundlehren der
  Mathematischen Wissenschaften.
\newblock Springer-Verlag, Berlin, 1992.

\bibitem{palezzato2020freehyperplanearbfield}
E.~Palezzato and M.~Torielli.
 \newblock Free hyperplane arrangements over arbitrary fields.
 \newblock {\em Journal of Algebraic Combinatorics}, 52(2): 237--249, 2020.

\bibitem{palezzato2018hyperplane}
E.~Palezzato and M.~Torielli.
\newblock Hyperplane arrangements in {C}o{C}o{A}.
\newblock {\em Journal of Software for Algebra and Geometry}, 9(1):43--54, 2019.

\bibitem{palezzato2020lefschetz}
E.~Palezzato and M.~Torielli.
 \newblock Lefschetz properties and hyperplane arrangements.
 \newblock {\em Journal of Algebra}, 555: 289--304, 2020.
 
 \bibitem{palezzato2020localiz}
E.~Palezzato and M.~Torielli.
 \newblock Localization of plus-one generated arrangements.
 \newblock {\em Communications in Algebra}, 49(1), 301--309, 2021.

\bibitem{schenck2002lower}
H.~Schenck and A.~Suciu.
\newblock Lower central series and free resolutions of hyperplane arrangements.
\newblock {\em Transactions of the American Mathematical Society},
  354(9):3409--3433, 2002.
  
\bibitem{suyama2019}
D.~Suyama, M.~Torielli and S.~Tsujie.
\newblock Signed graphs and freeness of the {W}eyl subarrangements of type $B_l$.
\newblock {\em Discrete Mathematics}, 342(1), 233--249, 2019.

\bibitem{stanley2004introduction}
R.~Stanley.
\newblock An introduction to hyperplane arrangements.
\newblock {\em Geometric combinatorics}, 389--496, IAS/Park City Math. Ser., 13, Amer. Math. Soc., 2007.

\bibitem{tortsujie2020}
M.~Torielli and S.~Tsujie.
\newblock Freeness of hyperplane arrangements between boolean arrangements and {W}eyl arrangements of type $B_l$.
\newblock {\em The Electronic Journal of Combinatorics}, 27(3), 2020.

\bibitem{tsujie2020}
S.~Tsujie.
\newblock Modular construction of free hyperplane arrangements.
\newblock {\em SIGMA}, 16, 080, 2020.

\bibitem{zaslavsky1989biased}
T.~Zaslavsky.
\newblock Biased graphs {I}. {B}ias, balance, and gains.
\newblock {\em Journal of Combinatorial Theory, Series B}, 47(1):32--52, 1989.

\bibitem{zaslavsky1991biased}
T.~Zaslavsky.
\newblock Biased graphs {II}. {T}he three matroids.
\newblock {\em Journal of Combinatorial Theory, Series B}, 51(1):46--72, 1991.

\bibitem{zaslavsky2003biased}
T.~Zaslavsky.
\newblock Biased graphs {IV}. {G}eometrical realizations.
\newblock {\em Journal of Combinatorial Theory, Series B}, 89(2):231--297,
  2003.

\end{thebibliography}

\end{document}